\documentclass[12pt, twoside, reqno] {amsart}

\usepackage{amsfonts}
\usepackage{amssymb}
\usepackage{latexsym}
\usepackage{epsf,graphicx}
\usepackage{epsf,graphicx,latexsym,%
}

\usepackage{xcolor}

\newcommand{\be} {\begin{eqnarray}}
\newcommand{\ee} {\end{eqnarray}}
\newcommand{\bep} {\begin{eqnarray*}}
\newcommand{\eep} {\end{eqnarray*}}

\newcommand {\Hol}{\mathop{\rm Hol}\nolimits}

\newcommand {\Id}{\mathop{\rm Id}\nolimits}

\renewcommand {\Re}{\mathop{\rm Re}\nolimits}

\newcommand {\JJ}{\mathcal{J}}

\newcommand {\A}{\mathcal{A}}
\newcommand {\BB}{\mathcal{B}}

\newcommand{\R}{{\mathbb R}}

\newcommand{\C}{{\mathbb C}}

\newcommand {\D}{\mathbb{D}}

\newtheorem{remar}{Remark}[section]
\newtheorem{examp}{Example}[section]
\newtheorem{defin}{Definition}[section]
\newtheorem{corol}{Corollary}[section]

\newtheorem{theorem}{Theorem}[section]
\newtheorem{lemma}{Lemma}[section]

\newcommand{\rema}{\begin{remar}\rm}
\newcommand{\erema}{$\blacktriangleright$\end{remar}}

\newcommand{\exa}{\begin{examp}\rm}
\newcommand{\eexa}{$\blacktriangleright$\end{examp}}

\def\lwvec(#1 #2){\linewd 0.1
           \lvec(#1 #2)
           \linewd 0.05}

\begin{document}

\title{Geometric features of nonlinear resolvents in the unit disk}

\author[M. Elin]{Mark Elin}

\address{Department of Mathematics,
         Ort Braude College,
         Karmiel 21982,
         Israel}

\email{mark$\_$elin@braude.ac.il}

\author[F. Jacobzon]{Fiana Jacobzon}

\address{Department of Mathematics,
         Ort Braude College,
         Karmiel 21982,
         Israel}

\email{fiana@braude.ac.il}


\begin{abstract}
We study nonlinear resolvents of holomorphic generators of one-parameter semigroups acting in the open unit disk.
The class of nonlinear resolvents can be studied in the framework of geometric function theory because it consists of univalent functions.

 In this paper we establish distortion and covering results, find order of starlikeness and of strong starlikeness of resolvents. This provides that any resolvent admits quasiconformal  extension to the complex plane $\C$. In addition, we obtain some characteristics of semigroups generated by these resolvents.


\end{abstract}
\maketitle
\section{Preliminaries}\label{sect-intro}

Let $D$ be a domain in the complex plane $\C$.  Denote the set of holomorphic functions on $D$ by $\Hol(D,\C)$, and by $\Hol(D) := \Hol(D,D)$, the set of all holomorphic self-mappings of $D$.  In what follows we use the notation $D_r$ for the open disk of radius $r$, namely, $D_r:=\left\{z:\ |z|<r\right\}$. Also we denote $\D=D_1$, the open unit disk.

Let $\Omega$ be the subclass of $\Hol(\D)$ consisting of functions vanishing at the origin:
\begin{equation}\label{def-U}
\Omega=\{ \omega \in \Hol(\D):\ \omega(0)=0 \}.
\end{equation}
The identity mapping on $\D$ will be denoted by $\Id$.

 The aim of this paper is to study a class of univalent functions in the open unit disk that occurs in nonlinear dynamical systems. This class consists of nonlinear resolvents, which, by their definition, are solutions to some functional equations involving so-called generators.
Recall that a mapping $f\in\Hol(\D,\C)$ is called an (infinitesimal) generator if for every $z\in\D$ the Cauchy problem
\begin{equation}  \label{nS1}
\left\{
\begin{array}{l}
\frac{\partial u(t,z)}{\partial t}+f(u(t,z))=0    ,     \vspace{2mm} \\
u(0,z)=z,%
\end{array}%
\right.
\end{equation}%
has a unique solution $u=u(t,z)\in\D$ for all $t\geq 0$. In this case, the unique solution of \eqref{nS1} forms a semigroup of holomorphic self-mappings of the open unit disk $\D$ generated by $f$; see, for example, \cite{B-C-DM-book, E-R-Sbook, E-S-book, R-S1, SD}. It turns out that generators can be characterized as follows.

 \begin{theorem}[see \cite{R-S1, SD, E-R-Sbook, E-S-book} for detail] \label{teorA}
 Let $f\in \Hol(\D , \C),\  f\not\equiv0$. The following statements are equivalent:
 \begin{enumerate}
 \item[(i)] $f$ is a generator on $\D$;

\item[(ii)] there exist a point $\tau\in \overline\D$ and a function $p\in\Hol(\D,\C)$ with ${\Re p(z)\ge0}$ such that
\begin{equation}\label{b-p}
f(z)=(z-\tau )(1-z\overline{\tau })p(z),\quad z\in\D;
\end{equation}

\item[(iii)] $f$ satisfies the so-called range condition:
\[
\left(\Id +rf\right)(\D)\supset\D\qquad    \mbox{for all }\quad r>0,
\]
and $G_r:=(\Id+rf)^{-1}$ is a well-defined self-mapping of $\D$.
\end{enumerate}
 \end{theorem}

We notice that formula \eqref{b-p} is known as the {\it Berkson--Porta representation} after the seminal work \cite{B-P} by Berkson and Porta. The mappings $G_r\in\Hol(\D),\ r>0,$ are called the {\it nonlinear resolvents} of the generator $f$, the net $\{G_r\}_{r>0}$ is the {\it resolvent family} for $f$. These are the main objects of the study in this paper.

Numerous properties of nonlinear resolvents can be found in the books \cite{SD, R-S1, E-R-Sbook}. In particular, the solution of the Cauchy problem~\eqref{nS1} can be reproduced by the following exponential formula:
\begin{equation}\label{expo-f}
u(t,\cdot) = \lim_{n\rightarrow \infty } \left(G_{\frac{t}{n}}\right)^{[n]},
\end{equation}
where $G^{[n]}$ means the $n$-th iterate of a self-mapping $G$ and the limit exists in the topology of uniform convergence on compact subsets of $\D.$

Since the representation \eqref{b-p} in Theorem \ref{teorA} is unique, it follows that every generator must have at most one null point in $\D $. This point $\tau $ is known to be the \textit{Denjoy--Wolff point} for the semigroup  $\left\{u(t,\cdot)\right\} _{t\geq 0}$ defined by \eqref{nS1} as well as for the resolvent family $\{G_r\}_{r>0}$. More precisely, if the function $p$ in assertion (ii) of Theorem~\ref{teorA} satisfies $\Re p(z) > 0$, then
\begin{equation}\label{D-Wp}
\tau= \lim\limits_{t\rightarrow \infty } u(t,z) =\lim_{r\to\infty}G_{r}(z) .
\end{equation}
Moreover, both these limits attain not only pointwise, but uniformly on compact subsets of the open unit disk.

\vspace{2mm}
In this paper we concentrate on the  case $\tau=0$.  Thus $f(z)=zp(z)$ with $\Re p(z)>0$  by the Berkson--Porta formula~\eqref{b-p}  and $\lim\limits_{r\to\infty} G_r(z)=0,$ uniformly on compact subsets of $\D$ by formula~\eqref{D-Wp}. In this situation there is a simply verifiable  condition providing that the convergence of the semigroup to its Denjoy--Wolff point $\tau=0$ is uniform on the {\it whole disk}~$\D$:

\begin{theorem}[see \cite{FSG, E-S-S, E-R-Sbook}]\label{thm_kappa}
Let $\kappa>0$ be a constant. The semigroup $\{u(t,\cdot)\}_{t\ge0}$ generated by~$f$, $f(z)=zp(z),$ satisfies the estimate $|u(t,z)|\le |z|e^{-\kappa t }$ for all $t>0$ and $z\in\D$ (and consequently $u(t,z)\to0$ as $t\to\infty$ uniformly on $\D$)  if and only if $\Re p(z)\ge\kappa,\ z\in\D$.
\end{theorem}
The number $\kappa$ in this theorem is called {\it exponential squeezing coefficient}.

To present another property of semigroups, which will appear below, we recall that semigroups by their definition are well-defined for real non-negative values of the parameter $t$ only. However, it may happen that for every fixed $z\in\D$, the semigroup $\{u(\cdot,z)\}$ can be   analytically extended to some  sector in the complex plane. Analyticity of semigroups was recently studied in \cite{A-C-P, E-J-17, E-S-Ta}, see also \cite[Chapter 6]{E-R-Sbook}. The following fact will be used in the sequel.

\begin{theorem}\label{thm-analyt}
  Let $\alpha,\beta\in(0,\frac\pi2)$. The semigroup $\{u(t,\cdot)\}_{t\ge0}$ generated by~$f$, $f(z)=zp(z),$ can be analytically extended to the sector $\{t\in\C: \ \arg t\in(-\alpha,\beta)\}$ for all $z\in\D$ if and only if $-\frac\pi2+\alpha<\arg p(z)<\frac\pi2-\beta,\ z\in\D$.
\end{theorem}
In this connection we also notice that due to the exponential formula~\eqref{expo-f}, the analyticity of a semigroup in some sector follows from the analyticity of all the resolvents in the same sector; see \cite[Section~6.2]{E-R-Sbook}.

 \vspace{2mm}

Besides their importance in dynamical systems,  nonlinear resolvents form a class of self-mappings of $\D$, which is interesting itself from the point of view of geometric function theory. For instance, it can be easily seen that nonlinear resolvents are univalent in the open unit disk.

Significant results concerning additional geometric aspects of nonlinear resolvents of generators vanishing at zero and such that $f'(0)>0$ were discovered recently in \cite{E-S-S}. To remind some of them and to present our subsequent results, recall several important classes of univalent functions intensively studied in geometric function theory.

\begin{defin}\label{def-starlike}
  Let $h,\ h(0)=0,$ be a univalent function in the open unit disk $\D$. Then $h$ is
  \begin{itemize}
\item [] \hspace{-6mm}starlike of order $\alpha\in(0,1) $ if $\Re\left(\frac{zh'(z)}{h(z)}\right)>\alpha$ for all $z\in\D$;
\item [] \hspace{-6mm}$\theta$-spirallike of order $\alpha\in(0,1) $ if $\Re \left(e^{-i\theta}\frac{zh'(z)}{h(z)}\right)>\alpha \cos \theta$ for all $z\in\D$;
\item []  \hspace{-6mm}strongly starlike of order $\beta\in(0,1)$ if $\left|\arg\frac{zh'(z)}{h(z)}\right|<\frac{\pi\beta}{2}$ for all $z\in\D$;
\item [] \hspace{-6mm}hyperbolically convex if $h \in \Hol(\D)$ and ${\Re\left(\frac{zf''(z)}{f'(z)}+1+\frac{2z\overline{f(z)}f'(z)}{1-|f(z)|^2}\right)>0.}$
  \end{itemize}
\end{defin}
Starlike and spirallike functions are classical objects of the study in geometric function theory. Actually, $\theta$-spirallike functions with $\theta=0$ are starlike.   The reader can be referred to the monograph \cite{G-K} and reference therein (see also \cite{Ke-Me}).
 As to hyperbolically convex functions, see, for example, \cite{Ma-Mi-94}. It was proved in \cite{Me-Pomm} that every hyperbolically convex function is starlike of order $\frac12$\,.

Among the results proved in \cite{E-S-S} we mention here that:
\begin{itemize}
  \item Any resolvent $G_r$ is a hyperbolically convex self-mapping of $\D$ and, consequently, a  starlike function of order $\frac12\,.$
  \item Any resolvent $G_r$ satisfies $\Re \frac{G_r(z) }{z}>\frac{1}{2(1+rf'(0))}$ and consequently, by Theorem~\ref{teorA}, is a generator. Moreover, by Theorem~\ref{thm_kappa}, the semigroup generated by $G_r$ converges to $0$ uniformly on~$\D$ with an exponential squeezing coefficient $\kappa =1/[2(1+rf'(0))] $.
  \item  If a generator $f$ is a starlike function of order $\alpha>\frac12$, then any resolvent $G_r$ extends to a $(\sin\pi\alpha)$-quasiconformal mapping of $\C$.
      \end{itemize}

In this paper, we present a general approach to the study of nonlinear resolvents in the framework of geometric function theory. On this way we establish distortion and covering results. Next, we supplement and enhance the above results from \cite{E-S-S} and extend them to a wider class of generators. Namely, we consider resolvents of generators $f$ such that $f(0)=0$ and $\Re f'(0)>0$. Denoting $q:=f'(0)$, we prove the following results.
\begin{itemize}
\item [$\star$] Any function $G_r,\ r>\frac6{\Re q}\,,$ is $\theta$-spirallike of order greater than $\frac{(r\Re q)^2\cos \theta-6r\Re q}{(r\Re q)^2 \cos \theta -36 \cos \theta}$ for any $\theta$ with $|\theta| \leq \arccos \frac{6}{r\Re q}$. Also $G_r$ is strongly starlike of order less than $\frac2\pi \arcsin \frac{6}{r\Re q}\,$. Hence $G_r$ admits $\frac{6}{r\Re q}\,$-\,quasiconformal  extension to $\C$ with no additional condition on $f$.

  \item [$\star$] The resolvent family $\{G_r\}_{r>0}$ converges to $0$ uniformly on~$\D$ as ${r \to \infty}$, while the net of normalized resolvents $\{(1+rq)G_r\}_{r>0}$ converges to the identity mapping  as $r\to \infty$, uniformly on compact subsets of $\D$.
  \item[$\star$]  For the semigroup generated by  $G_r,\ r>\frac6{\Re q}\,,$ we discover an exponential squeezing coefficient $\kappa$ as a function depending on $r$. In the particular case where $q$ is real, $\kappa =1/[2^{B(r)}(1+rq)] $, where $B(r)\to0$ as $r\to\infty$.
  \item[$\star$]  Also, we state that this semigroup admits analytic extension with respect to the parameter  $t$ to the sector of opening $\pi\,\frac{r\Re q-6}{r\Re q+6}\,.$
   \end{itemize}

\medskip

\section{Auxiliary results}\label{sect-Pre}
\setcounter{equation}{0}

Let $\alpha, \beta \in\C$ with $\Re \alpha \overline{\beta} >0.$ Consider the class $\A_{\alpha,\beta}$ consisting of functions holomorphic in the open unit disk $\D$ and satisfying the following conditions
  \begin{equation}\label{ineq}
 \A_{\alpha,\beta}=\!\left\{ F: F(0)=F'(0)-\beta=0,\,  \Re \frac1\alpha\left(\frac{F(z)}{ z}-\beta\right) >-\frac 1 2\!\right\}\!.
  \end{equation}
Note that if we define
\begin{equation*}
\psi(z)=\beta+\frac{\alpha z}{1-z}\,,
\end{equation*}
then
\begin{equation}\label{classA}
\A_{\alpha,\beta}:=\{F \in \Hol(\D,\C): \frac{F(z)}{z}\prec\psi\}.
\end{equation}
The subordination relation $ \frac{F(z)}{z}\prec \psi(z)$ means that $F \in \A_{\alpha,\beta}$ if there exists a function $\omega \in \Omega$ such that
\begin{equation*}
F(z)=z\psi(\omega(z))\quad \text{ for all } z\in \D.
\end{equation*}
Clearly, every $F \in \A_{\alpha,\beta}$ is locally univalent at the origin and the inverse function $F^{-1}$ satisfies $F^{-1}(0)=0$. Let also denote
\begin{equation}\label{classB}
\BB_{\alpha,\beta}:=\{F^{-1} : F\in \A_{\alpha,\beta}\}.
\end{equation}

\medskip

First we establish the radius of univalence for  the class $\BB_{\alpha,\beta}$ as well as covering and distortion results.

     \begin{theorem}\label{th_posi1}
For $\alpha, \beta \in\C$ with $\Re \alpha \overline{\beta} >0, $   denote $M=1-\Re\frac{\beta}{\alpha}\,$.
   Every function $G\in\BB_{\alpha,\beta}$ is univalent  in the disk $D_R$, where
 \[
  R= \left\{ \begin{array}{lc}
                |\alpha|\left(\frac{1}{2}-M\right) , & \  \mbox{if } \Re \frac{\beta}{\alpha}>\frac34\,, \vspace{2mm} \\
                  |\alpha|\left(1-\sqrt{M}\right)^2  \,, &  \ \mbox{if }   \Re \frac{\beta}{\alpha} \le \frac34\,,
                \end{array}
 \right.
 \]
and satisfies $G(D_{R}) \subset D_{ R_1}$ with
\[
 R_1= \left\{ \begin{array}{ccc}
                 1, & \quad \mbox{if } \Re \frac{\beta}{\alpha}>\frac34\,, \vspace{2mm}\\
                  \frac{1}{\sqrt{M}}-1 , &  \quad\mbox{if } \Re \frac{\beta}{\alpha}\le\frac34\, \vspace{2mm}.
                \end{array}
 \right.
 \]
 Further, $G(D_{R}) \supset D_{ R_2},$ where $R_2=\displaystyle\frac{RR_1}{R_1|\beta| + \sqrt{R_1^2|\beta|^2-R^2}}\,.$
  \end{theorem}

 \begin{proof} Let us represent the number $\frac{\alpha}{2\beta}$ in the form $se^{-i\theta}$, or $\alpha=2\beta se^{-i\theta}$, where $|\theta|<\frac\pi2$ since ${\Re\alpha\overline{\beta}>0}$. Then by \eqref{ineq}, we  get  $ \Re e^{i\theta}\frac{F(z)}{\beta z} >\cos\theta-s.$ Therefore the embedding  $ G(D_{ R}) \subset D_{ R_1}$ follows from \cite[Corollary 3.2]{E-S-2020a}. Moreover, it follows from the proof of Theorem~3.1 in \cite{E-S-2020a} that $G$ is a univalent function.

 It remains to prove the covering relation $D_{R_2} \subset G(D_{R})$ only.
 It follows from Corollary 3.2 in  \cite{E-S-2020a} that $G$ maps $D_R$ onto a hyperbolically convex subdomain of the disk $D_{R_1}$. Therefore, the function $h$ defined by $h(z)=\displaystyle\frac{G(Rz)}{R_1}\ $ belongs to $\Hol(\D)$ and is hyperbolically convex. By a result in \cite{Me-Mi-91} (see also \cite[Theorem~2]{Ma-Mi-94}), the image $h(\D)$ contains the disk of radius $\displaystyle\frac{|h'(0)|}{1+\sqrt{1-|h'(0)|^2}}\,$. Since $h'(0)=\frac{G'(0)R}{R_1}=\frac{R}{\beta R_1}\,,$ one concludes that $G(D_R)$ contains the disk of radius
 \[
    \frac{|h'(0)|R_1}{1+\sqrt{1-|h'(0)|^2}} = \frac{RR_1}{R_1|\beta| + \sqrt{R_1^2|\beta|^2-R^2}} = R_2.
 \]
 The proof is complete.
 \end{proof}

 \begin{remar}
  In the proof of this theorem, the hyperbolic convexity was used for one purpose only, namely, to prove the covering result. It is worth mentioning that some covering result can be obtained without this property.
 Indeed, let $w\not\in G(D_{R}) $. Then the function $h$ defined by $h(z)=\displaystyle\frac{G(Rz)}{w-G(Rz)}$ is holomorphic and univalent in the unit disk. It can be easily seen that
   \[
   h(0)=0,\quad h'(0)=\frac{G'(0)R}w\quad\mbox{and}\quad h''(0)=\frac{G''(0)w + 2G'(0)^2}{ w^2}R^2.
   \]
   Since by the famous Bieberbach theorem $|h''(0)|\le 4|h'(0)|$ (see, for example, \cite{G-K}) we have
   \[
   4|G'(0)w|\ge R\left| G''(0)w+2G'(0)^2\right| \ge 2R|G'(0)|^2-R|G''(0)w|.
   \]
     It follows from \cite[Proposition~4.1]{MF-new} that $G'(0)=\frac1\beta$ and $|G''(0)|\le \frac{|\alpha|}{|\beta|^3}$. This leads to $|w|\ge \displaystyle\frac{|\beta|R}{|\beta|^2 +|\Re \beta|R}$.
 \end{remar}

\begin{examp}\label{ex-starlike}
  Consider the set of all functions $F\in\Hol(\D,\C)$ such that $F(0)=F'(0)-1=0$ and $\Re\frac{F(z)}{z}\ge\frac12\,.$ This is equivalent to  $F\in \A_{1,1}$, that is, to the choice  $\alpha=\beta=1\,.$   Theorem~\ref{th_posi1} implies that every $G\in\BB_{1,1}$ is univalent in the disk of radius $\frac{1}{2}$ and  $D_{\frac{1}{2+\sqrt{3}}}\subset G(D_{\frac{1}{2}})\subset\D.$
\end{examp}

\medskip

\section{Main results}\label{sec-main}
\setcounter{equation}{0}

In this section we first apply the conclusion of Theorem~\ref{th_posi1} to  the special case where $\Re\beta>1$ and $\alpha =2(\Re\beta-1)$.  In this case the function $\psi(z)=(\Re\beta-1)\frac{2z}{1-z}+\beta$ maps the open unit disk onto the half-plane ${\{w:\,\Re w>1\}}$.
Holding in mind our interest in resolvents, we intend to study the net of such functions $\{\psi_r\}_{r>0}$ choosing $\beta$ to be an affine function of the positive parameter~$r$. Namely, from now on $q\in\C$ will be a fixed number with $\Re q>0$ and
\begin{equation}\label{psi-r}
\psi_r(z)=1+r\,\frac{q+\overline{q}z}{1-z}=1+rq+2r\Re q\sum_{n=1}^\infty z^n.
\end{equation}
Accordingly $\A_r:=\{F \in \Hol(\D,\C): \frac{F(z)}{z}\prec\psi_r\}$, cf. \eqref{classA}.
Next we formulate criteria for a holomorphic function $F$ to belong to the class~$\A_r$.
\begin{lemma}
  Let $F\in\Hol(\D,\C),\ F(0)=0.$ Then the following conditions are equivalent:
  \begin{itemize}
  \item [(i)] $\Re \frac{F(z)}{z}\ge1$ for all $z\in\D$ and $F'(0)=1+rq$;
  \item [(ii)] $F(z)=z+rz\frac{q+\overline{q}\omega(z)}{1-\omega(z)}$ for some $\omega\in\Omega$;
  \item [(iii)] $F\in \A_r$.
\end{itemize}
\end{lemma}
This lemma follows directly from our notations. Each  condition of this lemma is equivalent to the fact that the function $f$ defined by
\begin{equation} \label{G*-1}
 f(z)=\frac{F(z)-z}r=z\frac{q+\overline{q}\omega(z)}{1-\omega(z)}\,,\  \omega\in\Omega,
\end{equation}
is a generator on $\D$ by the Berkson--Porta formula. Hence the (right) inverse function $F^{-1}=:G(=G_r)$, which, in fact, solves the functional equation
\begin{equation} \label{G*-2}
G_r+ rf\circ G_r=\Id
\end{equation}
is holomorphic in the open unit disk $\D$ by Theorem~\ref{teorA}. (Recall that $G_r$ is a univalent self-mapping of $\D$, which is called   the resolvent of $f$, see Section~\ref{sect-intro}.)

In what follows, we focus on the family $\JJ_r:=\BB_{2r\Re q,1+rq},$ see \eqref{classB}, consisting of the resolvents $G_r=(\Id + rf)^{-1}$ with a fixed $q$.

\vspace{2mm}

 As we have already mentioned, the resolvent family converges to  the null point of $f$  as $r \to \infty$, uniformly on compact subset of the unit disk. We now show that for $r>\frac{2}{\Re q}$, nonlinear resolvents admit  forced analytic continuation to a disk of prescribed radius, and prove distortion and covering results. This enables us to establish that $G_r$ tends to $0$ as $r\to\infty$ uniformly on $\D$
.

 \begin{theorem}\label{th_resol1}
Let $r>\frac2{\Re q}\,.$
\begin{itemize}
  \item [(a)] Every element $G_r$ of $\JJ_r$ can be extended as a univalent function to the disk $D_{\rho(r)},\  \rho(r)={\displaystyle \left(\sqrt{2r \Re q}-\sqrt{r \Re q -1}\right)^2>1}$, and satisfies $G_r(D_{ \rho(r) }) \subset D_{ \rho_1(r)}$ with $\rho_1(r)= \displaystyle\sqrt{\frac{2r\Re q}{r\Re q-1}}  -1.$
  \item [(b)] $G_r(D_{\rho(r)})\supset D_{\rho_2(r)}$, where $\rho_2(r)=\displaystyle\frac{\left(\sqrt{2r \Re q}-\sqrt{r \Re q -1}\right)^2} {|1+rq| + \sqrt{2+r\Re q+r^2|q|^2}}\,.$
  \item [(c)] For any $r>0$ we have $G_r(\D)\supset D_{\rho_3(r)},$ where\\ $\displaystyle \rho_3(r)=\frac1{|1+rq| + \sqrt{|1+rq|^2 -1}}\,$.
\end{itemize}
 \end{theorem}

\begin{corol}\label{corr-uiform-conv}
 For any $r>\frac2{\Re q}$  we have
 $$|G_r(z)|\le\displaystyle\frac{3}{1+r\Re q}\quad \text{ for all } \quad z\in\D.$$
Thus  $G_r$ has no boundary fixed points and the net $\{G_r\}_{r>0}$ converges to zero uniformly on $\D$ as $r\to\infty$.
\end{corol}

We prove Theorem~\ref{th_resol1} and Corollary~\ref{corr-uiform-conv} simultaneously.

\begin{proof}
  It follows from \eqref{psi-r}, $\beta= 1+rq$ and $\alpha= 2r\Re q.$  The condition $r>\frac{2}{\Re q}$ is equivalent to $\Re\frac\beta\alpha<\frac34\,$. Following our notation $M=1-\Re\frac{\beta}{\alpha}=\frac{r\Re q-1}{2r\Re q}>\frac14\,.$ We now substitute these expressions in Theorem~\ref{th_posi1} and then get the univalence of $G_r$ in $D_{\rho(r)}$ as well as the inclusion   $G_r(D_{ \rho(r) }) \subset D_{ \rho_1(r)}$.

 In fact, this inclusion implies Corollary~\ref{corr-uiform-conv}. Indeed, consider the function $h$ defined by $h(z):=\frac{G_r(\rho(r)z)}{\rho_1(r)}\,$. Assertion (a) implies that $h$ is a self-mapping of the open unit disk and hence by the Schwarz lemma $|h(z)|\le|z|$, or equivalently, $|G_r(\rho(r)z)|\le \rho_1(r)|z|$ for all $z\in\D$. Denote $\zeta=\rho(r)z$. Then $|G_r(\zeta)|\le \frac{\rho_1(r)|\zeta|}{\rho(r)}$ for $\zeta\in D_{\rho(r)}.$ Since $\rho(r)>1$, one can take, in particular, $\zeta\in\D$ and this yields
\[
|G_r(\zeta)|\le \frac{\rho_2(r)}{\rho(r)} = \frac{1}{(r\Re q-1)\left( \sqrt{\frac{2r\Re q}{r\Re q-1}} -1 \right)} \le\frac{3}{1+r\Re q}\,.
\]

As to the covering result $G_r(D_{\rho(r)})\supset D_{\rho_2(r)}$, it follows from Theorem~\ref{th_posi1} as well.

  To complete the proof, we recall that according to a result in \cite{E-S-S} (see also \cite{E-S-2020a}) every resolvent is a hyperbolically convex function. In addition, it follows from \cite {Me-Mi-91} (see also \cite[Theorem 2]{Ma-Mi-94}) that the image of the unit disk under every hyperbolically convex function $h$ normalized by $h(0)=0,\ h'(0)=\delta\in(0,1),$ contains the disk of radius $\frac\delta{1+\sqrt{1-\delta^2}}$. Since $G_r'(0)=\frac1{1+rq}\,$, we obtain assertion (c).  This completes the proof.
\end{proof}

It was proved in~\cite{E-S-S} that for every $r>0$ the resolvent $G_r$ is a starlike function of order $\frac{1}{2}$\,. In the next theorem we describe the range of the function $\frac{wG'_r(w)}{G_r(w)}$, which provides an essentially stronger result. More precisely, we get order of starlikeness and order of strong starlikeness (see Definition~\ref{def-starlike}) of the resolvent $G_r$ as functions depending on the resolvent parameter $r$.

To this end we address to the function
 \begin{equation}\label{Ar}
 A(r):=\displaystyle\frac{6r (1+r )}{(1+r)^3- 3(5r-1) }
 \end{equation}
 and to the largest root $r_0$ of the equation $A(r)=1.$ It can be calculated that $r_0 = 1+2\sqrt{7}\cos\left(\frac{1}{3}\arctan\frac{3\sqrt{31}}{8}\right) \approx 5.92434\ldots$.

\begin{theorem}\label{th-r-star}
Let $G_r\in \JJ_r$, where $r\Re q> r_0$. Then for all $w\in\D,$
\begin{equation*}
\left|\frac{wG'_r(w)}{G_r(w)} - \frac{1}{1-A^2(r\Re q)}\right| \le \frac{A(r\Re q)}{1-A^2(r\Re q)}\,.
\end{equation*}
\end{theorem}

\begin{proof}
Denote  $p(z)=\frac{q+\overline{q}\omega(z)}{1-\omega(z)}\,$, cf. \eqref{G*-1}. Then \eqref{G*-2} implies
\begin{equation}\label{starGr}
\frac{wG'_r(w)}{G_r(w)}=\frac{1+rp\circ G_r(w)}{1+r\left(p\circ G_r(w)+ G_r(w) \cdot p' \circ G_r(w)\right)}\,.
\end{equation}
According to Corollary~\ref{corr-uiform-conv}, the inequality $|G_r(w)|\leq\frac{3}{1+r\Re q}$ holds for all $w\in \D$. Thus our aim is to find the range of  $\displaystyle\frac{1+rp(z)}{1+r\left(p(z)+ z  p' (z)\right)}$ whenever $|z|\leq \frac{3}{1+r\Re q}$.

By the Berkson--Porta representation~\eqref{b-p}, $\Re p \geq 0$.
The Riesz--Herglotz formula gives
\begin{equation*}
p(z)=\int_{|\zeta|=1} \frac{1+z\overline{\zeta}}{1-z\overline{\zeta}}\,d\mu(\zeta)+i\gamma,
\end{equation*}
for some non-negative measure $\mu$ on the unit circle and a number $\gamma \in \R,$ so that
\[
p(0)=q=\int_{|\zeta|=1}d\mu(\zeta)+i\gamma.
\]
Denote also $B_r(z):=\Re (1+rp(z))$,  $C_r(z):=\left| rzp'(z) \right|$ and
\begin{equation*}\label{ineqA1}
 A_r(z,\zeta):=\frac{2 r\Re q|z\overline{\zeta}|}{1+r\Re q- 2\Re z\overline{\zeta} + |z\overline{\zeta}|^2 (1-r\Re q) }\,.
\end{equation*}
One can see that for all $z$ with $|z|<\frac{3}{1+r\Re q}$ and $\zeta$ with $|\zeta|=1$,
\begin{equation*}\label{ineqAr}
 A_r(z,\zeta)\leq A_r\left(\frac3{1+r\Re q},1 \right) =  A(r\Re q).
\end{equation*}
Also,
\begin{equation*}
B_r(z)=\int_{|\zeta|=1} \frac{1+r\Re q- 2\Re z\overline{\zeta} + |z\overline{\zeta}|^2 (1-r\Re q) }{\Re q |1-z\overline{\zeta}|^2}d\mu(\zeta).
\end{equation*}
Thus
\begin{eqnarray*}
C_r(z) &\leq& \int_{|\zeta|=1} \frac{2 r|z\overline{\zeta}|}{|1-z\overline{\zeta}|^2}d\mu(\zeta) \\
\nonumber   &=& \int_{|\zeta|=1}  A_r(z,\zeta)\frac{1+r\Re q- 2\Re z\overline{\zeta} + |z\overline{\zeta}|^2 (1-r\Re q) }{\Re q |1-z\overline{\zeta}|^2}d\mu(\zeta)\\
\nonumber  &\leq& A(r\Re q)B_r(z).
\end{eqnarray*}
Therefore,
\begin{equation*}
 \left|  \frac{rzp'(z)} {1+rp(z)}\right| \leq \frac{C_r(z)} {B_r(z)}\leq  A(r\Re q).
\end{equation*}

Thus the function $\displaystyle\frac{1+r\left(p(z)+ z  p' (z)\right)}{1+rp(z)}$ takes values in the disk centered at $1$ and of radius $A(r\Re q)$. A straightforward calculation based on formula~\eqref{starGr} shows that all of the values of $\displaystyle\frac{wG'_r(w)}{G_r(w)}$ belong to the disk centered at $\displaystyle \frac{1}{1-A^2(r\Re q)}$ and of radius $\displaystyle \frac{A(r\Re q)}{1-A^2(r\Re q)}\,.$ The proof is complete.
\end{proof}

This theorem describes the range of the function $\frac{wG'_r(w)}{G_r(w)}$. Considering this range from the point of view of Definition~\ref{def-starlike}, we get the following geometric conclusion.
\begin{corol}
Let $G_r\in \JJ_r$, where $r\Re q> r_0$ and $\theta \in \R$ with $|\theta| \leq \arccos \frac{6}{r\Re q}$.
Then $G_r$ is a $\theta$-spirallike function of order $$\displaystyle \alpha_{r,\theta}:=\frac{\cos \theta -A(r\Re q)}{(1-A^2(r\Re q))\cdot \cos \theta}\,.$$
Consequently, $G_r$ is starlike of order $\displaystyle \alpha_r:=\frac1{1+A(r\Re q)}$ and strongly starlike of order $\beta_r:=\displaystyle\frac2\pi \arcsin A(r\Re q)$.
\end{corol}

\begin{remar}\label{rem-1}
One can see that if $r\Re q > 6$ then $A(r) < \frac{6}{6-r_0+r}< \frac6r$. This enables to estimate orders of spirallikeness, starlikeness and strong starlikeness. Namely, the order of $\theta$-spirallikeness $\alpha_{r,\theta}$ is greater than  $\frac{r\Re q(r\Re q\cos \theta-6)}{((r\Re q)^2  -36)\cos \theta }$,
order of starlikeness $\alpha_r$ is greater than  $\frac{6-r_0+r\Re q}{12-r_0+r\Re q}>\frac{r\Re q }{6+r \Re q }\,,$ order of strong starlikeness $\beta_r$ is less than $\frac2\pi\arcsin\frac6{6-r_0+r\Re q}< \frac2\pi \arcsin \frac{6}{r\Re q}\,.$
\end{remar}

Recall that a homeomorphism $h: \C\to\C$ is called $k$-quasiconformal if $h$ has locally integrable partial derivatives on $\C$ (in the sense of distributions), which satisfy $\left|f'_{\overline{z}}\right| \le k\left|f'_z\right|$ a.e. It was proved in \cite{F-K-Z} (see also \cite{Su-12}) that any strongly starlike function of order $\alpha$ extends to a $\sin(\pi\alpha/2)$-quasiconformal automorphism of $\C$. Therefore, Theorem~\ref{th-r-star} immediately implies
\begin{corol}
  Any function $G_r\in\JJ_r,\ r\Re q> r_0$, can be extended to a $k$-quasiconformal  automorphism of \,$\C$ with  $k=A(r\Re q)$ the following.
\end{corol}

\medskip

Every resolvent $G_r$ is actually a generator itself, see \cite{E-S-S}. Therefore it is natural to ask about properties of semigroups generated by nonlinear resolvents. In the next assertion we establish the uniform convergence (on the whole disk $\D$) of such semigroups as well as their analyticity in a sector  with respect to the parameter $t$.

\begin{theorem}\label{thm-estim}
  Let $G_r\in\JJ_r$ with $r\ge \frac6{\Re q}\,$. Denote $\gamma_r:=\frac{1-A(r\Re q)}{1+A(r\Re q)}$, where function $A$ is denoted by \eqref{Ar}. Then for the semigroup $\{u(t,\cdot)\}_{t\ge0}$ generated by $G_r$ the following assertions hold:
  \begin{itemize}
    \item [(i)]  $u(t,\cdot)$  converges to $0$ as $r\to\infty$, uniformly on $\D $ with exponential squeezing coefficient $\kappa(r) :=\displaystyle\frac{\left( \Re(1+rq)^\frac1{\gamma_r} \right)^{\gamma_r}} {2^{1-\gamma_r}|1+rq|^2}\,;$
    \item [(ii)]  for every $z\in\D$, $u(\cdot,z)$ can be analytically extended to the sector $$\left\{t\in\C: \left|\arg t - \arg(1+rq)\right| < \frac{ \pi\gamma_r}{ 2}\right\}.$$
    \end{itemize}
  \end{theorem}

\begin{proof}
  Since the function $G_r$ is starlike of order $\alpha_r$ by Theorem~\ref{th-r-star}, the function $(1+rq)G_r$ is a normalized starlike function of the same order, and hence admits the integral representation
\[
(1+rq)G_r(z)=z \exp\left[ -2(1-\alpha_r)\oint_{\partial\D} \log\left( 1-z\overline{\zeta}\right)d\mu_r(\zeta) \right]
\]
  with some probability measure $\mu_r$ on the unit circle.
  Therefore the function $z\left(\frac{(1+rq)G_r(z)}z\right)^{\frac1{2(1-\alpha_r)}}=z\left(\frac{(1+rq)G_r(z)}z\right)^{\frac1{1-\gamma_r}}$ is starlike of order $\frac12$\,. Consequently,
 \begin{equation}\label{aux1}
  \Re \left(\frac{(1+rq)G_r(z)}z\right)^{\frac1{1-\gamma_r}}>\frac12\,.
  \end{equation}

According to Theorem~\ref{thm_kappa}, to prove assertion (i), we have to show that $ \Re \frac{G_r(z)}{z} > \kappa(r).$
Let denote for short $w=\frac{G_r(z)}z$ and $B(r)=1-\gamma_r$.
Our aim is to minimize $\Re w$ under the condition $\Re \left((1+rq)w\right)^{\frac1{B(r)}} = \frac12\,,$ see \eqref{aux1}. In other words, we have to minimize the function $\zeta(t):=\Re\frac{\left(\frac12+it\right)^{B(r)}}{1+rq}\,.$ Equating $\zeta'(t)$ to zero, we get $\arg \left(\frac12+it\right)^{1-B(r)} = \arg(1+r\overline q)$ at the minimal point of $\zeta$, or equivalently, $\frac12+it = M(1+r\overline q)^{\frac1{1-B(r)}}$ for some $M>0.$ This leads to $\frac12+it =\frac{(1+r\overline q)^{\frac1{1-B(r)}}}{2\Re (1+rq)^{\frac1{1-B(r)}}}\,. $ Thus
\begin{eqnarray*}
  \min_{t\in\R} \zeta(t) &=& \Re\frac{(1+r\overline q)^{\frac{B(r)}{1-B(r)}}} {(1+rq)2^{B(r)} \left(\Re(1+rq)^{\frac1{1-B(r)}} \right)^{B(r)}} \\
    &=&  \frac{\left( \Re(1+rq)^\frac1{1-B(r)} \right)^{1-B(r)}} {2^{B(r)}|1+rq|^2}=\kappa(r),
\end{eqnarray*}
and we are done.

To prove assertion (ii), we use the same notations. We now estimate the values of $\arg w$. In other words, values of the function $\xi(t):=\arg\frac{\left(\frac12+it\right)^{B(r)}}{1+rq} = B(r)\arg\left(\frac12+it\right) - \arg(1+rq).$ Obviously, $$\xi(t)\in \left(-\frac\pi2 B(r)-\arg(1+rq),\frac\pi2 B(r)-\arg(1+rq) \right).$$
Applying Theorem~\ref{thm-analyt}, we complete the proof.
 \end{proof}

 Similarly to Remark \ref{rem-1}, we note that $\gamma_r>\frac{r\Re q-r_0}{12-r_0+r\Re q}>\frac{r\Re q-6}{r\Re q +6}\,.$

\begin{corol}\label{cor-tends_to_z}
The net of functions  $\left\{(1+rq)G_r(z)\right\}$ converges to $z$  as $r\to \infty$, uniformly on compact subsets of the unit disk.
\end{corol}

\end{document}